\documentclass[a4paper, 12pt, notitlepage]{amsart}
\usepackage{latexsym, amssymb, amscd, amsfonts, amsthm, amsmath, graphicx,
mathabx, mathrsfs}

\newcommand{\GL}{\operatorname{GL}}
\newcommand{\SL}{\operatorname{SL}}

\newcommand{\GCD}{\operatorname{GCD}}

\newcommand{\bC}{\mathbb{C}}

\newcommand{\bR}{\mathbb{R}}

\newcommand{\zed}{\mathbb{Z}}

\newcommand{\sgn}{\mathrm{sgn}}

\newcommand{\SO}{\mathrm{SO}}

\newcommand{\sE}{\mathscr{E}}

\newcommand{\sL}{\mathscr{L}}
\newcommand{\sR}{\mathscr{R}}
\newcommand{\sS}{\mathscr{S}}

\newcommand{\sX}{{\mathscr{X}}}

\newtheorem{theorem}{Theorem}

\newtheorem{lemma}[theorem]{Lemma}
\newtheorem{proposition}[theorem]{Proposition}

\theoremstyle{remark}
\newtheorem*{rem}{Remark}

\title{Maass form twisted Shintani $\sL$-functions}
\author{Bob Hough}
\address{School of Mathematics, Institute for Advanced Study,  1 Einstein
Drive,
Princeton, NJ, 08540}
\email{hough@math.ias.edu}


\thanks{This material is based upon work supported by the National Science
Foundation under agreement No.\ DMS-1128155. Any opinions, findings and
conclusions or recommendations expressed in this material are those of the
author and do not necessarily reflect the views of the National Science
Foundation.}

\thanks{This work was partially supported by ERC Grant 279438,
Approximate Algebraic Structure and Applications.}

\thanks{This work was partially supported by the AIM
Square on alternative proofs of the Davenport-Heilbronn Theorems. }

\begin{document}

\begin{abstract}
The Maass-form twisted Shintani $\sL$-functions are introduced, and some of
their analytic properties are studied.  These functions contain data
regarding the distribution of shapes of cubic rings.
\end{abstract}

\maketitle

\section{Introduction}
The space of binary cubic forms over a commutative ring  $\sR$
\begin{equation}
 V_{\sR} = \left\{ f(x,y) = ax^3 + bx^2y + cxy^2 + dy^3: a,b,c,d \in \sR
\right\}
\end{equation}
has a rich algebraic structure. $\GL_2(\sR)$ acts
by changing coordinates:
\begin{equation}
g\cdot f(x,y) = f((x,y)g^t).
\end{equation}
 Over $\bC$, this makes $V_{\bC}$ an example of a
prehomogeneous vector space. Over $\bR$, $V_{\bR}$  splits into
a pair of
open $\GL_2(\bR)$ orbits, having positive
and negative discriminant, and a singular set
having discriminant 0. The non-singular forms have finite stabilizer, so that
these are naturally identified with finite quotients
of $\GL_2(\bR)$. Over $\zed$, one considers in addition to the
lattice $L=V_\zed$, the dual lattice
\begin{equation}
 \hat{L} = \left\{f \in L: 3| b, c \right\}.
\end{equation}
For a fixed non-zero $m \in \zed$ those integral forms from $L$ and
$\hat{L}$ of discriminant $m$ each split into finitely many orbits, the
number of which is the class number, denoted $h(m)$ and $\hat{h}(m)$,
respectively. The
space of integral binary cubic forms taken modulo $\GL_2(\zed)$-equivalence
has extra significance, as it is
in discriminant-preserving bijection with cubic rings taken up to isomorphism
\cite{DF64}, \cite{GGS02}, \cite{B04}.

 Shintani \cite{S72} introduced zeta functions enumerating the
class numbers $h(m), \hat{h}(m)$.  These Dirichlet series, initially defined
only in the half-plane $\{s \in \bC: \Re(s)>1\}$, have meromorphic continuation
to all of $\bC$ and satisfy a functional equation relating $s$ to $1-s$.
Shintani determined the poles and residues, and hence obtained strong results
on the average behavior of $h(m)$.

Having fixed a base point, a class of integral forms $\tilde{f}$ of non-zero
discriminant is identified with a point in $\GL_2(\zed)\backslash\GL_2(\bR)$,
and it is natural to ask for the distribution of these points on average.  We
call the point $g_f \in \GL_2(\zed) \backslash\GL_2(\bR)$ the `shape' of the
form $f$, a name which becomes more natural in the case that $f$ is associated
to an order of a cubic field, in which case $g_f$ describes the shape of the
corresponding lattice in its natural embedding. The distribution of these shapes
was studied by Terr \cite{T97}, who proved the asymptotic uniform distribution
of the shape of cubic orders and fields when ordered by discriminant, see also \cite{H15}.
In related work, the
author proved the quantitative equidistribution of 3-torsion ideal classes in
imaginary quadratic fields \cite{H16} and the corresponding equidistribution
statements for quartic and quintic fields have been demonstrated by
Bhargava and Harron \cite{BH13}.

The purpose of this note is to give a strong estimate for the
equidistribution of binary cubic forms
with respect to the cuspidal spectrum of $\SL_2(\bR)/\SL_2(\zed)$, by modifying
the method of Shintani. 
Let $\phi$ be a non-constant automorphic cusp form on
\begin{equation}\sX = \SO_2(\bR)\backslash\SL_2(\bR)/\SL_2(\zed),\end{equation}
which is an eigenfunction of the Hecke algebra, and extend $\phi$ to
$\GL_2(\bR)$
by projecting by a diagonal matrix.  Fix base forms $x_{\pm}^0$ of discriminant
$\pm 1$, and for each $m \neq 0$ choose representatives
$\left\{g_{i,m}\right\}_{i=1}^{h(m)},
\left\{\hat{g}_{i,m}\right\}_{i=1}^{\hat{h}(m)}$ such that
$\left\{g_{i,m}\cdot x_{\sgn(m)}^0\right\}_{i=1}^{h(m)}$,
$\left\{\hat{g}_{i,m}\cdot
x_{\sgn(m)}^0\right\}_{i=1}^{\hat{h}(m)}$ are representatives for the classes of
integral forms of discriminant $m$.   Denote $\Gamma(i,m)$, $\hat{\Gamma}(i,m)$
the stability
groups of $g_{i,m}\cdot
x_{\sgn(m)}^0$, resp. $\hat{g}_{i,m}\cdot
x_{\sgn(m)}^0$ in $\Gamma = \SL_2(\zed)$.  Introduce
`$\phi$-twisted Shintani $\sL$-functions' defined for
$\Re(s) > 4$ by absolutely convergent Dirichlet series
\begin{align}
 \sL_{\pm}(L, s; \phi)&= \sum_{\pm m \geq 1}
\frac{1}{|m|^s}\sum_{i=1}^{h(m)}\frac{\phi\left(g_{i,m}^{
-1}\right)}{|\Gamma(i,m)|},\\
\notag \sL_{\pm}(\hat{L}, s; \phi) &=  \sum_{\pm m \geq 1}
\frac{1}{|m|^s}\sum_{i=1}^{\hat{h}(m)}\frac{
\phi\left(\hat{g}_{i,m}^{-1}\right)}{\left|\hat{\Gamma}(i,
m)\right| }.
\end{align}
It is shown that these series may be factored from an orbital integral as in \cite{S72}.  The trick which permits introducing $\phi$ is
due to Selberg \cite{S56}, exploiting the mean-value property of harmonic functions.

The twisted $\sL$-functions appear less natural than the case $\phi = 1$ of \cite{S72}.  For instance, we are
not aware that a functional equation is satisfied, and suspect that none
exists.  We are, however, able to demonstrate the holomorphic continuation past
the region of absolute convergence, which is sufficient to prove
 equidistribution statements.

\begin{theorem}\label{maass_form_theorem}
 Let $\phi$ be a Maass Hecke-eigen cusp form on $\sX$. The $\phi$-twisted
Shintani $\sL$-functions extend to holomorphic
functions in the half-plane $\Re(s) > \frac{1}{8}$.
\end{theorem}
\begin{rem}
Theorem \ref{maass_form_theorem} exhibits  substantial orthogonality
of the shapes of binary cubic forms to the Maass spectrum.  In particular, for $\psi \in C_c^\infty(\bR^+)$, the proof of Theorem \ref{maass_form_theorem} permits the estimate
\begin{equation}\label{weyl_sum}
 \sum_{\pm m \geq 1}
 \psi\left(\frac{|m|}{X}\right)\sum_{i=1}^{h(m)}\frac{\phi\left(g_{i,m}^{
-1}\right)}{|\Gamma(i,m)|} \ll_{\epsilon, \psi} X^{\frac{1}{8}+\epsilon}
\end{equation}
with the same estimate for dual forms. By comparison, the number of forms counted is order $X$. The best estimate in (\ref{weyl_sum}) obtainable from \cite{T97} is of order $X^{\frac{15}{16}}$, while \cite{BH13} proves the qualitative statement $o(X)$.
\end{rem}

\begin{rem}
Recall that a cusp form $\phi$ of
$\SO_2(\bR)\backslash\SL_2(\bR)/\SL_2(\zed)$ satisfies
an exponential decay condition in the cusp.  Our argument applies with
appropriate modifications also to the Eisenstein spectrum, and to automorphic
forms that transform on the left by a fixed character of $\SO_2(\bR)$. See
\cite{LRS09} for a general description of automorphic forms on
$\SL_2(\bR)/\SL_2(\zed)$.  We omit the details here, but intend to give detailed equidistribution statements in a future paper treating cubic fields.
\end{rem}

\subsection*{Related work}
We discovered the twisted $\sL$-functions during
work on the AIM Square on alternative proofs of the Davenport-Heilbronn
theorems.  See work of Sato \cite{S94}, \cite{S06} for some related
objects.

\section{Background}

Set $G = \GL_2(\bR)$, $G^1 = \SL_2(\bR)$, $G^+ = \{g \in G: \det g > 0\}$,
$\Gamma = \SL_2(\zed)$, $\Gamma_\infty = \Gamma \cap \begin{pmatrix}\pm 1 &0\\ *
& \pm 1\end{pmatrix}$ and standard subgroups\footnote{$c(\theta) = \cos(2\pi
\theta)$, $s(\theta) = \sin(2\pi \theta)$}
\begin{align}
 K &= \left\{k_\theta = \begin{pmatrix} c(\theta)& s(\theta)\\ -s(\theta)&
c(\theta)\end{pmatrix}: \theta \in \bR/\zed\right\},\\
\notag A &= \left\{a_t = \begin{pmatrix} t &0\\ 0 & \frac{1}{t}\end{pmatrix}: t
\in
\bR_{>0}\right\},\\
\notag N &= \left\{n_u = \begin{pmatrix} 1&0\\ u &1\end{pmatrix}: u \in
\bR\right\}.
\end{align}
Haar measure is normalized on $G^1$ by setting, for $f \in L^1(G^1)$,
\begin{align}
 \int_{G^1}f(g)dg 
 &= \int_{\bR/\zed} \int_{\bR^+}\int_\bR f(k_\theta a_t n_u)d\theta
\frac{dt}{t^3}du
\end{align}
and, for $f \in L^1(G)$,
\begin{equation}
 \int_{G^+} f(g)dg = \int_{\bR^+} \int_{G^1}f\left(\begin{pmatrix}\ell &0\\
0&\ell \end{pmatrix}g\right)dg \frac{d\ell}{\ell}.
\end{equation}

\subsection{Automorphic forms}
For consistency with Shintani we work on $L^2(K\backslash G^1/\Gamma)$, with the
lattice quotient on the right.  This differs from many modern authors. See \cite{Y11} for a summary of the results discussed here, and note the normalization $y = t^2$.

A convenient basis for $L^2(K\backslash G^1/ \Gamma)$ consists in joint
eigenfunctions of the Laplacian and the Hecke operators.  These
automorphic forms split into discrete and continuous spectrum.  The discrete
spectrum has an $L^2$ basis of Hecke-eigen Maass forms while the continuous
spectrum is spanned by the real analytic Eisenstein series.

Let $\phi(g)$ be a Hecke-eigen Maass form with Laplace eigenvalue $\lambda =
s(1-s)$, $s = \frac{1}{2} + i t_\phi$.  The Maass forms split into even and
odd forms.  An even Maass form $\phi$ has a Fourier development in the parabolic
direction
\begin{equation}\label{maass_fourier_dev}
 \phi(g) = 2t \sum_{n=1}^{\infty}\rho_\phi(n) K_{s-\frac{1}{2}}\left(2\pi n
t^2\right)\cos(2\pi nu)
\end{equation}
whereas an odd form replaces $\cos(\cdot)$ with $\sin(\cdot)$ in the Fourier
expansion.
We use the Mellin transforms
\begin{align}
 \int_0^\infty K_\nu(x)x^{s-1}dx &= 2^{s-2}\Gamma\left(\frac{s+\nu}{2}
\right)\Gamma\left(\frac{s-\nu}{2} \right), \qquad \Re s > |\Re \nu|,\\
 \notag \int_0^\infty \cos(x)x^{s-1}dx &= \Gamma(s)\cos\left(\frac{\pi s}{2}
\right), \qquad 0 <\Re s < 1.
\end{align}
We assume the Maass forms considered are even, although the
argument applies to odd forms without change.
Let the Maass forms  be Hecke-normalized, that is, $\rho_\phi(1) =1$.
This means that the Fourier coefficients satisfy the Hecke relations
\begin{equation}
 \rho_\phi(m)\rho_\phi(n) = \sum_{d | \GCD(m,n)}\rho_{\phi}\left(\frac{mn}{d^2}
\right),
\end{equation}
from which it follows that there exists constant $C>1$ such that for all primes
$p$ and $n \geq 1$,
\begin{equation}
 |\rho_\phi(p^n)| \leq \left(C(1 + |\rho_\phi(p)|)\right)^n.
\end{equation}
The sup bound
\begin{equation}
 |\rho_\phi(n)| \ll n^{\frac{7}{64}+\epsilon}
 \end{equation}
was proven in \cite{KS03} while the $L^2$-bound
\begin{equation}
 \sum_{n \leq X}\left|\rho_\phi(n)\right|^2 \ll X
\end{equation}
follows from Rankin-Selberg theory.

We follow Shintani's convention regarding the real analytic Eisenstein series,
which puts the symmetry line for these forms at $\Re(z) = 0$.  Define
\begin{equation}\label{def_real_analytic_eisenstein}
E(z,g) = \sum_{\gamma \in \Gamma/\Gamma_\infty}
t(g \gamma)^{z+1}
\end{equation}
 the real analytic Eisenstein series with complex parameter
$z$.  This satisfies a functional equation
\begin{equation}
\xi(z + 1)E(z,g) = \xi(1-z)E(-z,g); \qquad \xi(z) =
\pi^{-\frac{z}{2}}\Gamma\left(\frac{z}{2}\right)\zeta(z)
\end{equation}
and has a Fourier development in $z \neq 0$ given by
\begin{align}\label{eisenstein_fourier_dev}
E(z,g)&= t^{z+1}+ t^{1-z}\frac{\xi(z)}{\xi(z+1)}\\\notag & \qquad+
\frac{4t}{\xi(z+1)} \sum_{m=1}^\infty
\eta_{\frac{z}{2}}(m)K_{\frac{z}{2}}(2\pi m t^2)\cos 2\pi m u,\\\notag
\eta_{\frac{z}{2}}(m)&=\sum_{ab = m}\left(\frac{a}{b} \right)^{\frac{z}{2}}.
\end{align}

Say that $f \in C(K\backslash G^1/\Gamma)$ is of polynomial growth if $f$ is
bounded by
a polynomial in $\theta, u, t$, similarly, is Schwarz class if it decays when
multiplied by any polynomial in $\theta, u, t$. The Maass forms are Schwarz
class, while the Eisenstein series has polynomial growth.  After subtracting
the constant term in the Fourier expansion, the resulting modified Eisenstein
series again is Schwarz class.

Due to convergence issues resulting from the constant term it is convenient to
work with a truncated Eisenstein
series.  Let $\Psi$ denote the space of entire
functions such that for all $\psi \in \Psi$, for all $-\infty < C_1 < C_2 <
\infty$, for all $ N >0$,
\begin{equation}
 \sup_{C_1 < \Re (w) < C_2}\left(1 + (\Im w)^2\right)^N |\psi(w)| < \infty.
\end{equation}
For $\psi \in \Psi$ and $\Re(w) > 1$ define the incomplete Eisenstein series
at $\psi$ by choosing $1 < c < \Re (w)$ and setting
\begin{equation}
\sE(\psi, w; g) = \oint_{\Re (z) = c}\psi(z)\frac{E(z,g)}{w-z}dz.
\end{equation}

\subsection{Binary cubic forms}
$G$ acts naturally on the space
\begin{equation}
V_\bR = \left\{ax^3 + bx^2y + cxy^2 + dy^3: (a,b,c,d) \in
\bR^4\right\}
\end{equation}
of binary cubic forms via, for $f \in V_{\bR}$ and $g \in G$,
\begin{equation}
 g\cdot f(x,y) = f((x,y)\cdot g^t).
\end{equation}
The discriminant
$D$, which is a homogeneous polynomial of degree four on $V_{\bR}$, is a
relative
invariant:  $D(g\cdot f) = \chi(g)D(f)$
where $\chi(g) = \det(g)^6$. One identifies the dual space of $V_{\bR}$ with
$\bR^4$ via alternating
pairing \begin{equation}
\langle x, y \rangle = x_4y_1 - \frac{1}{3}x_3y_2 + \frac{1}{3}x_2y_3 -
x_1y_4.
\end{equation}
Let $\tau$ be the map $V_{\bR}\to V_{\bR}$ carrying each basis vector to its
dual basis
vector; the discriminant $\hat{D}$ on the dual space is normalized such that
$\tau$ is discriminant preserving.  There is an involution $\iota$ on $G$ given
by
\begin{equation}
 g^\iota = \begin{pmatrix} 0&-1\\ 1 &0\end{pmatrix} (g^{-1})^t \begin{pmatrix}
0 & 1\\ -1 &0\end{pmatrix}.
\end{equation}
This satisfies, for all $g \in G$, $x \in V_\bR$, $y \in \hat{V}_\bR$,
\begin{equation}
 \langle x, y \rangle = \langle g\cdot x, g^\iota \cdot y \rangle.
\end{equation}

The set of forms of zero discriminant are called the singular set, $S$. The
non-singular forms split into spaces $V_+$ and $V_-$ of positive and negative
discriminant.  The space $V_+$ is a single $G^+$ orbit
with representative $x_+ = (0,1,-1,0)$ and stability group
\begin{equation}
 I_{x_+} = \left\{I, \begin{pmatrix}0&1\\-1&-1\end{pmatrix}, \begin{pmatrix}-1
&-1\\ 1&0\end{pmatrix} \right\}.
\end{equation}
Set $x_+^0 = \lambda_+ x_+$, rescaled to have discriminant 1.
$V_-$ is also a single $G^+$ orbit with representative $x_- = (0,1,0,1)$ with
trivial stabilizer. $x_-^0 = \lambda_- x_-$ is also rescaled to have
discriminant 1.  

Set $w_1 = (0,0,1,0),  w_2 = (0,0,0,1).$
The singular set is the disjoint union
\begin{equation}
 S = \{0\}\sqcup G^1 \cdot w_1 \sqcup G^1 \cdot w_2.
\end{equation}
The stability group for the action of $G^1$ on $w_1$ is trivial $I_{w_1} =
\{1\}$, while on $w_2$ it is
$ I_{w_2}= N.$

Over $\zed$ write $L$ and $\hat{L}$ for the lattices of integral forms and
their dual, and write $L_0$ and $\hat{L}_0$ for those integral forms, resp.
dual forms, of discriminant zero. $L_0$
and $\hat{L}_0$ are the disjoint unions
\begin{align}
 L_0 &= \{0\}\sqcup L_0(I) \sqcup L_0(II), \qquad \hat{L}_0 &=
\{0\}\sqcup L_0(I)\sqcup \hat{L}_0(II),
\end{align}
with
\begin{align}L_0(I) &= \bigsqcup_{m=1}^\infty \bigsqcup_{\Gamma/\Gamma\cap
N}\gamma
\cdot (0,0,0,m),\\
\notag L_0(II) &= \bigsqcup_{m=1}^\infty
\bigsqcup_{n=0}^{m-1}\bigsqcup_{\gamma \in \Gamma} \gamma \cdot (0,0,m,n),\\
\notag \hat{L}_0(II) &= \bigsqcup_{m=1}^\infty
\bigsqcup_{n=0}^{3m-1}\bigsqcup_{\gamma \in \Gamma}\gamma \cdot (0,0,
3m,n).
\end{align}

Given Schwarz function $f \in \sS(V_{\bR})$ one has the Fourier transforms
\begin{align}
 \hat{f}(x) &= \int_{V_\bR}f(y)e(\langle x, y\rangle)dy, \qquad  f(x ) =
\frac{1}{9}\int_{V_\bR}\hat{f}(y)e(\langle x,y \rangle)dy.
\end{align}
For $\ell \in \bR_{>0}$ write $f_\ell(x) = f(\ell x)$. Say that $f$ is
left-$K$-invariant if, for all $x \in V_{\bR}$, for all $k \in K$, $f(k\cdot
x)= f(x)$.  One easily checks that $f$ and $\hat{f}$ are simultaneously
left-$K$-invariant.  Say that $f$ is right-$K$-invariant if, for both choices of
$\pm$, for all $g \in G^+$, for all $k \in K$, $f(gk\cdot x_{\pm}) = f(g \cdot
x_{\pm})$.  Let $\phi$ be a Maass form and let $f$ be right-$K$-invariant.
Identify $f_\ell\left(g
\cdot x_{\pm}^0\right)$ as functions $f_{\ell, \pm}$
on $G^1/I_{x_{\pm}}$.
Interpret, for $h \in G^1$,
\begin{equation}
 \int_{G^1}
f_\ell\left(g \cdot x_{\sgn \;m}^0\right)\phi\left(gh\right)dg
\end{equation}
as group convolution on $G^1$, written $\check{f}_{\ell, \pm}*\phi(h)$. The
result
obtained is left-$K$-invariant.  Since the Laplacian and Hecke operators
commute with translation, it follows by multiplicity 1 that $\check{f}_{\ell,
\pm} *
\phi =
\lambda(f_{\ell, \pm}, \phi) \phi$ is a scalar multiple times $\phi$.

\section{Dirichlet series}
Let $\phi \in C( K \backslash G^1/\Gamma)$ be a Hecke-eigen Maass form and
extend $\phi$ to $G$ by projecting
onto $G^1$. Note that this means that  $\phi(g) =
\phi(g^\iota)$ since $g$ and $g^{\iota}$ differ by a scalar.
Let $f \in \sS(V_{\bR})$.
Adapting Shintani's construction, introduce orbital integrals
\begin{align}
 Z(f, L; s, \phi) &= \int_{ G_+/\Gamma}\chi(g)^s
\phi\left(g\right)\sum_{x
\in L \setminus L_0}f(g\cdot x)dg\\ \notag
 Z(f, \hat{L}; s, \phi) &= \int_{ G_+/\Gamma}\chi(g)^s
\phi\left(g\right)\sum_{x
\in \hat{L} \setminus \hat{L}_0}f(g\cdot x)dg.
\end{align}

\begin{lemma}
 Let $f \in \sS(V_{\bR})$ be right-$K$-invariant.  Let
$\phi$ be a
Maass form satisfying, for $\ell >0$,
\begin{equation}
\check{f}_{\ell, \pm} *
\phi =
\lambda(f_{\ell, \pm}, \phi) \phi.
\end{equation}
For $\Re(s)$ sufficiently large, the orbital integrals $ Z(f, L; s, \phi), Z(f,
\hat{L}; s, \phi)$ satisfy
\begin{align}
Z(f, L; s, \phi)
=&\sL_+(L, s; \phi)\int_{0}^\infty \lambda(f_{\ell, +},
\phi)\ell^{12s-1}d\ell \\ \notag& + \sL_-(L, s; \phi)\int_{0}^\infty
\lambda(f_{\ell,
-},
\phi)\ell^{12s-1}d\ell\\ \notag
Z(f, \hat{L};s,\phi)
=&\sL_+(\hat{L}, s; \phi)\int_{0}^\infty \lambda(f_{\ell, +},
\phi)\ell^{12s-1}d\ell \\& \notag+ \sL_-(\hat{L}, s; \phi)\int_{0}^\infty
\lambda(f_{\ell,
-},
\phi)\ell^{12s-1}d\ell.
 \end{align}
\end{lemma}

\begin{proof}
One finds for
$\Re(s)$ sufficiently large
\begin{align}
 &Z(f, L; s, \phi) = \int_{G_+/\Gamma}\sum_{m \neq
0}\sum_{i=1}^{h(m)}\frac{\chi(g)^s
\phi\left(g\right)}{|\Gamma(i,m)|} \sum_{\gamma \in
\Gamma}f( g\gamma g_{i,m} \cdot x_{\sgn \; m}^0 )dg\\
\notag&= \sum_{m \neq
0}\frac{1}{|m|^s}\sum_{i=1}^{h(m)}\frac{\ell^{12s-1}}{|\Gamma(i,m)|} \int_{0}^{\infty}\int_ { G^1
}
f_\ell\left(g \cdot x_{\sgn \;m}^0\right)\phi\left(gg_{i,m}^{-1}\right)dg d\ell
\\\notag&= \sum_{m \neq 0}
\frac{1}{|m|^s}\sum_{i=1}^{h(m)}\frac{\phi\left(
g_{i,m}^{-1}\right)}{|\Gamma(i,m)|}\int_{0}^\infty
\lambda(f_{\ell, \sgn \; m},
\phi)\ell^{12s}\frac{d\ell}{\ell}\\
\notag&=\sL_+(L, s; \phi)\int_{0}^\infty \lambda(f_{\ell, +},
\phi)\ell^{12s}\frac{d\ell}{\ell} + \sL_-(L, s; \phi)\int_{0}^\infty
\lambda(f_{\ell,
-},
\phi)\ell^{12s}\frac{d\ell}{\ell}.
 \end{align}
The proof for the dual $\sL$-functions is the same.
\end{proof}

Following Shintani, introduce
\begin{align}
 Z^+(f, L; s, \phi) & = \int_{G_+/\Gamma, \chi(g)\geq 1}\chi(g)^s
\phi(g)\sum_{x \in L \setminus L_0}f(g\cdot x)dg,\\
\notag Z^+(f, \hat{L}; s, \phi) & = \int_{G_+/\Gamma, \chi(g)\geq 1}\chi(g)^s
\phi(g)\sum_{x \in \hat{L} \setminus \hat{L}_0}f(g\cdot x)dg.
\end{align}
These functions converge absolutely and are entire.

The following proposition is the analogue of \cite{S72}, Proposition 2.14.
\begin{proposition}
 For $\Re(s)> 4$,
 \begin{align}\notag
  &Z(f,L;s,\phi) = Z^+(f,L;s,\phi) + Z^+(\hat{f}, \hat{L}; 1-s, \phi)\\
\label{small_det}&
- \int_{\substack{G_+/\Gamma,\\ \chi(g) <
1}}\chi(g)^s \phi(g)\left\{\sum_{x \in L_0}f(g \cdot x) -
\chi(g)^{-1}\sum_{x \in \hat{L}_0}\hat{f}(g^\iota \cdot x)
\right\}dg\\ \notag
 &Z(f, \hat{L}; s, \phi) = Z^+(f,\hat{L};s,\phi) + \frac{1}{9}Z^+(\hat{f}, L;
1-s, \phi)\\ \label{dual_small_det}&
- \int_{\substack{G_+/\Gamma,\\ \chi(g) <
1}}\chi(g)^s \phi(g)\left\{\sum_{x \in
\hat{L}_0}f(g \cdot x) - \frac{1}{9}\chi(g)^{-1}\sum_{x \in
L_0}\hat{f}(g^\iota \cdot x) \right\}dg.
 \end{align}
\end{proposition}
\begin{proof}
 Write
 \begin{align}
 & Z(f,L;s,\phi) \\\notag&= \int_{ G_+/ \Gamma}\chi(g)^s
\phi(g)\sum_{x
\in L }f(g\cdot x)dg - \int_{ G_+/ \Gamma}\chi(g)^s
\phi(g)\sum_{x
\in L_0 }f(g\cdot x)dg.
 \end{align}
Split the first integral at $\det(g) \geq 1$.  In the integral with $\det(g)<
1$ perform Poisson summation in the sum over $L$, using that $F_g(x)= f(g\cdot
x)$ has $\hat{F}_g(y) = \frac{1}{\chi(g)}\hat{f}(g^\iota \cdot y)$. The proof
for $\hat{L}$ is similar.
\end{proof}

The objective now is to give the holomorphic continuation of (\ref{small_det})
and (\ref{dual_small_det}). This closely follows
the evaluation of Shintani leading up to the Corollary to Proposition 2.16 of
\cite{S72}.

Given $f \in \sS(V_{\bR})$ which is left-$K$-invariant, introduce
distributions,
for $z, z_1, z_2 \in \bC$ and $u \in \bR$,
\begin{align}
 \Sigma_1(f,z_1, z_2) &= \int_{0}^{\infty}\int_0^\infty
(f(0,0,t,u)+ f(0,0,t,-u))t^{z_1-1}u^{z_2-1} dt du\\
\notag \Sigma_2(f,z)&= \int_0^\infty f(0,0,0,u)u^{z-1}du\\
\notag \Sigma_3(f,z,u) &= \int_0^\infty f(0,0,t,u)t^{z-1}dt.
\end{align}

Following Shintani, for $g \in  G^1/\Gamma$ define
\begin{align}
 J_L(f)(g) &= \sum_{x \in L_0}f(g\cdot x), \qquad
 J_{\hat{L}}(f)(g) = \sum_{x \in \hat{L}_0} f(g \cdot x).
\end{align}
It follows from \cite{S72} Lemma 2.10 that for $f \in \sS(V_\bR)$, for $\phi$
of at most polynomial growth,
$\phi(g)J_L(f)(g)$ and $\phi(g)J_{\hat{L}}(\hat{f})(g)$ have at most polynomial
growth, while
\begin{equation}\phi(g)\left(J_L(f)(g) - J_{\hat{L}}(\hat{f})(g)
\right)\end{equation} is a Schwarz-class function on $G^1/\Gamma$.

The starting point is the formula (see e.g. \cite{S72}, p. 174)
\begin{align}
\label{Eisenstein_eval}&\frac{\psi(1)}{\xi(2)}\int_{
G^1/ \Gamma}\left(J_L(f)(g) -
J_{\hat{L}}(\hat{f})(g) \right) \phi(g)dg
\\\notag & = \lim_{w \downarrow 1}(w-1)
\int_{ G^1/\Gamma} \left(J_L(f)(g) -
J_{\hat{L}}(\hat{f})(g) \right) \phi(g) \sE(\psi, w; g) dg.
\end{align}
We have the following evaluation of integrals.
\begin{lemma}
 Let $\phi$ be a Maass form. Then \begin{align}
  \int_{G^1/\Gamma}\sE(\psi, w; g)\phi(g)dg &= 0.
  \end{align}
\end{lemma}

\begin{proof}
Let $1 < c < \Re(w)$. Opening $\sE(\psi, w; g)$ as a contour integral, then
unfolding the Eisenstein
series, one obtains
\begin{equation}
\int_{G^1/\Gamma} \sE(\psi,w;g)\phi(g)dg =
\frac{1}{2}\int_{G^1/\Gamma \cap N} \left( \oint_{\Re (z) =
c}\frac{t(g)^{1 +
z}}{w-z} \psi(z)dz \right) \phi(g) dg
\end{equation}
This integral now vanishes by integrating in the parabolic direction, since the
Maass form has no constant term.
\end{proof}
Let $\psi_1, \psi_2$ be two holomorphic functions in the half-plane $\Re(w)>4$.
 Say these functions are equivalent $\psi_1 \sim \psi_2$ if $(\psi_1 -
\psi_2)$ may be meromorphically continued to $\Re(w)> 0$ and is holomorphic in
a neighborhood of $w = 1$. Equivalent functions are interchangeable in the integrand of (\ref{Eisenstein_eval}).

Let $\phi \in C(K \backslash G^1/\Gamma)$.  Set
\begin{align}
\Theta_{\psi}^{(1)}(w;\phi) &= \int_{ G^1/\Gamma} \sE(\psi, w;
g)\phi(g)\sum_{x \in L_0(I)}f(g\cdot x) dg\\ \notag
\Theta_{\psi}^{(2)}(w;\phi)&= \int_{  G^1/\Gamma}\sE(\psi,
w;g)\phi(g)\sum_{x\in L_0(II)}f(g\cdot x)dg\\ \notag
\hat{\Theta}_\psi^{(2)}(w;\phi)&= \int_{  G^1/\Gamma}\sE(\psi,
w;g)\phi(g)\sum_{x\in \hat{L}_0(II)}f(g\cdot x)dg.
\end{align}
Also,  write $\phi_c(t)$ for its constant
term, found by integrating away the parabolic direction.
\begin{lemma}
Let $f \in\sS(V_\bR)$ be left-$K$-invariant.  Given Maass form $\phi$,
\begin{align}
 &\Theta_{\psi}^{(1)}(w;\phi) \sim 0.
\end{align}
\end{lemma}

\begin{proof}
 Let $\Re (w) > 2$.  Write
 \begin{align}
  \Theta_\psi^{(1)}(w;\phi)&= \int_{
G^1/\Gamma}\sum_{m=1}^\infty
\sum_{\Gamma/\Gamma \cap N}f( g\gamma\cdot (0,0,0,m))
\sE(\psi, w;
g)\phi(g)dg.
 \end{align}
 Introduce the Dirichlet series
 \begin{equation}
  F_\phi(u; z) = \sum_{n \geq 1}
\frac{\eta_{\frac{z}{2}}(n)\rho_\phi(n)}{n^u},
 \end{equation}
 which converges absolutely in $\Re(u) - \frac{|\Re(z)|}{2} > 1$.
After unfolding the sum over $\Gamma/\Gamma\cap N$ and integrating
in the compact and parabolic directions, this becomes (see \cite{S72} p. 178, middle display, for the first evaluation)
\begin{align}
 &\Theta_\psi^{(1)}(w;\phi)= \sum_{m=1}^\infty \int_0^\infty
f(0,0,0, t^{-3}m)\left(\oint_{\Re (z) = 2} \frac{\left(E(z;\cdot)\phi
\right)_c(t)}{w-z}\psi(z) dz\right)\frac{dt}{t^3}
\\ \notag&= 
4\sum_{m=1}^\infty \int_0^\infty
f(0,0,0, t^{-3}m)\\& \notag \times \left(\oint_{\Re (z) = 2} \frac{\left(
\sum_{n \geq 1} \eta_{\frac{z}{2}}(n)\rho_\phi(n)
K_{\frac{z}{2}}(2\pi n t^2)K_{s-\frac{1}{2}}(2\pi n
t^2)\right)}{\xi(z+1)(w-z)}\psi(z)
dz\right)\frac{dt}{t}\\ \notag
&= 4 \oiint_{\substack{\Re (u,z)\\  = (3,2)}} \Sigma_2(f,u)\int_0^\infty
\sum_{m=1}^\infty
\sum_{n=1}^\infty
\frac{\eta_{\frac{z}{2}}(n)\rho_\phi(n)}{\xi(z+1)(w-z)}\\\notag&
\qquad \times
\left(\frac{t^3}{m}
\right)^u K_{\frac{z}{2}}(2\pi n t^2)K_{s-\frac{1}{2}}(2\pi n t^2) \psi(z)
\frac{dt}{t}dzdu\\ \notag
&= 4\oiint_{\substack{\Re (u,z)\\  = (3,2)}}\int_0^\infty
\frac{\Sigma_2(f,u)\zeta(u)F_\phi\left(\frac{3u}{2};z\right)}{\xi(z+1)(w-z)}
\psi(z)
 t^{3u}
K_{\frac{z}{2}}(2\pi  t^2)K_{s-\frac{1}{2}}(2\pi  t^2) \frac{dt}{t}dzdu.
\end{align}
Shift the $z$ contour to $\Re(z) = 0$
to verify that
$\Theta_\psi^{(1)}(w;\phi)$ is holomorphic in $\Re(w) > 0$.

\end{proof}

Introduce
\begin{equation}
 G_{\phi}(x) = \sum_{\ell,m =
1}^{\infty}\frac{\rho_{\phi}(\ell m)}{\ell^{1 + x}m^{1 +
3x}}.
\end{equation}
Form $\hat{G}_\phi$ by dilating the sum over $m$ by 3.

\begin{lemma}
Given Maass form $\phi$,  $G_\phi(x)$ is holomorphic in
the
half-plane $\Re(x)> \frac{-1}{4}$.
\end{lemma}
\begin{proof}
Let $L_p(s, \phi) = \sum_{n \geq 0}\frac{\rho_\phi(p^n)}{p^{ns}}$ be the local
factor in the $L$-function $L(s,\phi) = \prod_p L_p(s, \phi)$ in $\Re(s)>1$.
 For $\Re(x)>-\frac{1}{4}$, write the local factor at prime $p$ in $G_\phi(x)$
as
\begin{align}
G_{\phi, p}(x)=& L_p(1+x, \phi)L_p(1 + 3x, \phi)\\\notag& \times\left(1 +
\frac{\rho_\phi(p^2)-\rho_\phi(p)^2}{p^{2 + 4x}} +
O\left(\frac{(1 + |\rho_\phi(p)|)^{3}}{p^{3 + 7x}}\right)\right).
\end{align}
It follows that $G_\phi(x) = L(1+x, \phi)L(1 + 3x, \phi)H_\phi(x)$ where
$H_\phi$ is given by an absolutely convergent Euler product in $x >
-\frac{1}{4}.$
\end{proof}

The Archimedean counterpart to $G_\phi$ is
\begin{align}
 W_\phi(w_1, w_2)&= \frac{2^{w_2-3}}{\pi^{\frac{1+w_1+w_2}{2}}}
\Gamma\left(1-w_2
\right)\cos\left(\frac{\pi}{2}(1-w_2)
\right)
\\&\notag\times \Gamma\left(\frac{-1+w_1+3w_2 +2it_\phi}{4}
\right)\Gamma\left(\frac{-1+w_1+3w_2 -2it_\phi}{4} \right)
\end{align}
which is holomorphic in $\left\{w_1, w_2: \Re(w_1 + 3w_2)>1,
\Re(w_2)<1\right\}$.

\begin{lemma}
Let $f \in\sS(V_\bR)$ be left-$K$-invariant.  Given Maass form $\phi$,
\begin{align}
 \Theta_{\psi}^{(2)}(w;\phi)
\sim  &\frac{\psi(1)}{\xi(2)(w-1)} \oiint_{\substack{\Re
(w_1,w_2) = (1, \frac{1}{2})}}
\Sigma_1(f,w_1,
w_2)\\\notag & \times W_{\phi}(w_1,w_2) G_\phi\left(\frac{w_1
+ w_2-1}{2}\right)  d w_1 dw_2.
\end{align}
To obtain the corresponding terms for $\hat{\Theta}_{\psi}^{(2)}(w;\cdot)$
replace $G$ with $\hat{G}$.
\end{lemma}

\begin{proof}
 Calculate (see \cite{S72}, p.179, next to last display)
 \begin{align}
   &\Theta_{\psi}^{(2)}(w;\phi)\\\notag &= \int_{G^1/\Gamma}\sE(\psi,
w;g)\phi(g)\sum_{m=1}^{\infty}\sum_{n=-\infty}^{\infty} \sum_{\gamma \in
\Gamma/\Gamma\cap N }f( g\gamma \cdot(0,0,m,n))dg\\
\notag &= \int_0^\infty \int_0^1 \sum_{m=1}^\infty
\sum_{n=-\infty}^{\infty}f(a_t\cdot
(0,0,m,n+mu))\oint_{\Re (z) =5}\frac{E(z,g)\phi(g)}{w-z}
\psi(z)dz du\frac{dt}{t^3}.
 \end{align}
In the Eisenstein series, separate the constant term, writing $\tilde{E}(z,g)
=E(z,g)- E(z,g)_c$.  The contribution of the non-constant part of
$\tilde{E}(z,g)\phi(g) $ is holomorphic in $\Re(w) > 0$ by
tracing  \cite{S72}, p. 180, top.

 The contribution of $(\tilde{E}(z,g)\phi(g))_c$ is given by
\begin{align}
 &\frac{4}{\xi(z+1)}\int_0^{\infty}\int_{-\infty}^{\infty}\oint_{\Re(z) =
5}\sum_{m=1}^\infty f(0,0,t^{-1}m,
u)\frac{\psi(z)}{w-z}\\\notag
&\qquad \times \left(\sum_{n=1}^\infty \rho_\phi(n)\eta_{\frac{z}{2}}(n)K_{\frac{z}{2}} (2\pi nt^2)K_{s-\frac{1}{2}}(2\pi
nt^2)\right) dzdu t^2dt\\\notag
&=
\frac{2}{\xi(z+1)}\int_0^{\infty}\int_{-\infty}^{\infty}\oiint_{\substack{\Re(z,
z') \\=
(5,5)}}\Sigma_3(f,z',u)\frac{\zeta(z')F_{\phi}(\frac{3+z'}{2};z)}{(2\pi
)^{\frac{3+z'}{2}}}  \frac{\psi(z)}{w-z}\\\notag
&\qquad \times K_{\frac{z}{2}} (t)K_{s-\frac{1}{2}}(t)t^{\frac{3 +
z'}{2}}dz' dz du
\frac{dt}{t}.
\end{align}
Shift the $z$ contour to $\Re(z)=0$ to verify that this is holomorphic in
$\Re(w)>0$.

From the constant term of $E(z,g)$,
only the term $\frac{\xi(z)}{\xi(z+1)}t^{1-z}$ contributes, and from this term
one picks up a pole at $z = 1$. Following Shintani, this yields
\begin{align}
  \Theta_{\psi}^{(2)}(w;\phi) &\sim \frac{2\psi(1)}{\xi(2)(w-1)}\int_0^\infty
\int_{-\infty}^\infty\sum_{\ell,m =1}^{\infty}\rho_\phi(\ell
m)\\\notag& \times K_{s-\frac{1}{2}}(2\pi \ell m t^2)f(0,0,t^{-1}m,u)\cos(2\pi
\ell t^3 u)du t dt .
\end{align}
Split the integral over $u$ by writing $f(0,0, *, u)= f_+(0,0,*,u)$ for $u > 0$
and $f(0,0, *, u) = f_{-}(0,0,*,-u)$ for $u < 0$.  Now open $f$ by taking
Mellin transforms in both variables,
\begin{align}
& \Theta_{\psi}^{(2)}(w;\phi) \sim \frac{2\psi(1)}{\xi(2)(w-1)}\int_0^\infty
\int_{0}^\infty  \oiint_{\substack{\Re(w_1, w_2)\\=(1, \frac{1}{2})}}
\Sigma_1(f,w_1, w_2)\sum_{\ell,m
=1}^{\infty} \frac{\rho_\phi(\ell m)}{m^{w_1}}\\ \notag
&\times K_{s-\frac{1}{2}}(2\pi \ell m t^2) \cos(2\pi \ell t^3
u)u^{-w_2}t^{1+w_1}d w_1 dw_2 du dt .
\end{align}
Replace $u:= 2\pi \ell t^3 u$, then $t:= 2\pi \ell m t^2$ to obtain
\begin{align} &\Theta_{\psi}^{(2)}(w;\phi)\sim \frac{\psi(1)}{\xi(2)(w-1)}
\oiint_{\substack{\Re(w_1, w_2)\\=(1, \frac{1}{2})}}
\Sigma_1(f,w_1, w_2)
\frac{G_\phi\left(\frac{w_1 + w_2-1}{2} \right)}{(2\pi )^{\frac{1 +
w_1+w_2}{2}}}\\
\notag&\times \int_0^\infty
\int_{0}^\infty K_{s-\frac{1}{2}}(t) \cos(
u)\frac{du}{u^{w_2}} \frac{dt}{t^{\frac{3-w_1-3w_2}{2}}}d w_1 dw_2
\\ \notag
&\sim \frac{\psi(1)}{\xi(2)(w-1)} \\
\notag&\times\oiint_{\substack{\Re(w_1, w_2)\\=(1, \frac{1}{2})}}
\Sigma_1(f,w_1,
w_2)W_\phi(w_1,w_2)G_\phi\left(\frac{w_1 + w_2-1}{2} \right) d w_1 dw_2.
\end{align}

\end{proof}

Putting together the above lemmas we conclude
\begin{align}
& \int_{ G^1/\Gamma}\left(J_L(f) - J_{\hat{L}}(\hat{f})
\right)\phi(g)dg=\\
&\notag    \oiint_{\substack{\Re(w_1, w_2)\\=(1,
\frac{1}{2})}}
\Sigma_1(f,w_1,
w_2) W_\phi(w_1, w_2)G_\phi\left(\frac{w_1
+ w_2-1}{2} \right) d w_1 dw_2
\\\notag &- \text{ terms replacing $f, G$ with $\hat{f}, \hat{G}$}.
\end{align}

We  now holomorphically extend the orbital integrals.  Note that
\begin{align}
\Sigma_1(f_t, w_1, w_2) &= t^{-w_1-w_2}\Sigma_1(f, w_1, w_2).
\end{align}

\begin{proof}[Proof of Theorem \ref{maass_form_theorem}]
Choose $f$ which is bi-$K$-invariant, and arrange $f$ such that the integrals
$\int_0^\infty \lambda(f_{\ell, \pm}, \phi)\ell^{12s}\frac{d\ell}{\ell}$ are
entire. Since $Z^+$ is entire, it suffices to consider the integral
 \begin{align}
&-\int_{G^+/\Gamma, \chi(g)\leq 1} \chi(g)^s
\phi(g)\left\{\sum_{x \in L_0}f(g\cdot x) -\chi^{-1}(g)\sum_{x \in
\hat{L}_0}\hat{f}(g^{\iota}x) \right\}dg\\
\notag &= -\int_0^1 t^{12s}\int_{G^1/\Gamma}\phi(g_1)\left\{\sum_{x \in
L_0}f_{t^3}(g_1 \cdot x) - \sum_{x \in \hat{L}_0}\widehat{f_{t^3}}(g_1 \cdot x)
\right\}dg_1 \frac{dt}{t}.
 \end{align}
 The contribution from $f$ may be expressed
 \begin{align}
&- \int_0^1 \oiint_{\Re(w_1, w_2)=(1,
\frac{1}{2})}t^{12 s - 3w_1 - 3w_2}
\Sigma_1(f,w_1,
w_2) \\&\notag\times W_\phi(w_1,w_2)G_\phi\left(\frac{w_1
+ w_2-1}{2} \right)  d w_1 dw_2
\frac{dt}{t}.
 \end{align}
Shift the $w_1$
contour left to $\Re(w_1) = \epsilon$. This expression is holomorphic in
$\Re(s)> \frac{1}{8}+\epsilon$. The contribution from $\hat{f}$ may be expressed (see \cite{S72}, p. 182)
\begin{align}
 &\int_0^1 \oiint_{\Re(w_1, w_2)=(1,
\frac{1}{2})}t^{12 s-12 + 3w_1 + 3w_2}
\Sigma_1(\hat{f},w_1,
w_2) \\&\notag\times W_\phi(w_1,w_2)\hat{G}_\phi\left(\frac{w_1
+ w_2-1}{2} \right)  d w_1 dw_2
\frac{dt}{t}.
\end{align}
In this integral, integration with respect to $w_1$ may be pushed right as far
as we like, so that the integral itself is holomorphic.
\end{proof}

\end{document}